\documentclass[10pt]{amsart}
\usepackage[margin=1in]{geometry}
\usepackage{url}
\usepackage{tikz-cd}
\usetikzlibrary{arrows}
\usetikzlibrary{calc}
\usepackage{hyperref}
\hypersetup{colorlinks=true, citecolor=blue}

\usepackage[lite]{amsrefs}
\usepackage{amssymb}
\usepackage{mathrsfs}
\usepackage[shortlabels]{enumitem}
\usepackage[all,cmtip]{xy}

\makeatletter
\@namedef{subjclassname@2020}{\textup{2020} Mathematics Subject Classification}
\makeatother

\newtheorem{pro}{Proposition}[section]
\newtheorem{theo}[pro]{Theorem}
\newtheorem{defn}[pro]{Definition}
\newtheorem{lem}[pro]{Lemma}

\newtheorem{ex}[pro]{Example}
\newtheorem{re}[pro]{Remark}

\begin{document}
	
\title[Cohomological Chow Groups of varieties with isolated singularities]{Cohomological Chow Groups of codimension one of varieties with isolated singularities}
\author[Diosel L\'opez-Cruz]{Diosel L\'opez-Cruz}
\address{Instituto de Matemáticas-CU, Universidad Nacional Autónoma de México. Ciudad de México 04510, México.}
\email{diosel@im.unam.mx}

\begin{abstract}
We compute some particular examples of cohomological Chow groups for varieties with isolated singularities. For higher-dimensional varieties, we compute the cohomological Chow groups of codimension one, provided that the dual complex associated to the normal crossing divisor is contractible. For 3-dimensional varieties, we consider a weaker condition on the dual complex, namely $H^{2}(\Gamma(E))=0$. 
\end{abstract}

\maketitle

\section{Introduction}
\noindent For a quasi-projective variety $X$ over a field $k$,  higher Chow groups are defined in terms of algebraic cycles. In \cite{BLO}, Bloch defines its {\it cycle complexes} $\mathcal{Z}^{r}(X, \bullet)$, and the higher Chow groups ${\rm CH}^{r}(X,m)$ as the $m^{\rm th}$-homology of these cycle complexes. A fundamental property of higher Chow groups is their relation to motivic cohomology:
$${\rm CH}^{r}(X, m) \otimes \mathbb{Q} \cong H^{2r-m}_{\mathcal{M}}(X, \mathbb{Q}(r)),$$ if $X$ is smooth  \cite[Theorem 19.1]{MVW}. In the singular case, higher Chow groups do not form a cohomology theory, but a Borel-Moore homology. A cohomological version is given as an iterative construction in terms of resolution of singularities (cubical hyperresolutions) and cycle complexes by Hanamura in \cite{Han}. Both theories coincide in the smooth case, but in general Hanamura's construction of cohomological higher Chow groups coincides with the Friedlander-Voevodsky motivic cohomology defined in terms of cdh-topology \cite{FV}. Let $X$ be a singular quasi-projective variety over a field $k$ that admits resolution of singularities. A semi-simplicial hyperresolution $a\colon X_{\bullet} \rightarrow X$ in the sense of \cite{GNPP} is a semi-simplicial scheme consisting of smooth schemes $X_{p}$ for $0 \leq p \leq n$ with some $n \in \mathbb{Z}_{\geq 0}$, together with morphisms $d_{p,i}\colon X_{p} \rightarrow X_{p-1}$ for $0 \leq i \leq p$, called face maps:
$$X_\bullet = \Bigl\{\begin{tikzcd}[column sep=2em]
	\cdots \ar[shift left=0.9em]{r}\ar[shift left=0.3em]{r}\ar[shift right=0.3em]{r}\ar[shift right=0.9em]{r} & X_2 \ar[shift left=0.6em]{r}{d_{2,0}}\ar{r}[description]{d}\ar[shift right=0.6em]{r}[swap]{d_{2,2}} & X_1 \ar[shift left=0.3em]{r}{d_{1,0}}\ar[shift right=0.3em]{r}[swap]{d_{1,1}} & X_0
\end{tikzcd}\Bigr\} \xrightarrow{a} X$$ such that $d_{p,i} \circ d_{p+1,i}=d_{p,i} \circ d_{p+1,i+1}$ with $0 \leq i \leq p$. The semi-simplicial hyperresolution induces a fourth quadrant double complex
\[ \begin{tikzcd}
	\mathcal{Z}^{r}(X_{\bullet}, \bullet) : \mathcal{Z}^{r}(X_{0}, \bullet) \ar{r}{d^{\ast}} & \mathcal{Z}^{r}(X_{1}, \bullet) \ar{r}{d^{\ast}} & \cdots \ar{r}{d^{\ast}} & \mathcal{Z}^{r}(X_{n}, \bullet)
\end{tikzcd} \] where in the $p$-th position we have Bloch's cycle complex of $X_{p}$, and the horizontal differentials $d^{\ast}:=\sum (-1)^{i}d^{\ast}_{i}\colon \mathcal{Z}^{r}(X_{p}, -q) \rightarrow \mathcal{Z}^{r}(X_{p+1}, -q)$ are the alternating sums of the pull-backs. The {\it cohomological cycle complex} $\mathcal{Z}^{r}(X_{\bullet}, \bullet)^{\ast}$ of $X$ is the total complex of the above double complex, and the {\it cohomological Chow group} (or {\it motivic cohomology}) is by definition the $m^{\rm th}$-homology of this complex: ${\rm CHC}^{r}(X,m)=H_{m}\big(\mathcal{Z}^{r}(X_{\bullet}, \bullet)^{\ast}\big)$. There is an associated spectral sequence
$$E^{p,q}_{1}(r):={\rm CH}^{r}(X_{p},-q) \Rightarrow {\rm CHC}^{r}(X, -p-q)$$ which does not depend of the choice of semi-simplicial hyperresolution $a\colon X_{\bullet} \rightarrow X$ \cite{Han}. Let $X$ be a normal variety of dimension $d \geq 3$ (then the singular locus $X_{\rm sing}$ is of codimension $\geq 2$). The resolution of singularities of $X$ along $X_{\rm sing}$ induces a commutative square
\[ \begin{tikzcd} 
E=p^{-1}(X_{\rm sing}) \ar{r} \ar{d} & \widetilde{X} \ar{d}{p} \\ X_{\rm sing} \ar{r} & X
\end{tikzcd} \] where $p$ is a birational proper map, and $E$ is a normal crossing divisor. The cohomological cycle complex of $X$ is:
$$\mathcal{Z}^{r}(X, \bullet)^{\ast}={\rm Cone}\big\{\mathcal{Z}^{r}(\widetilde{X}, \bullet) \oplus \mathcal{Z}^{r}(X_{\rm sing}, \bullet)^{\ast} \ \rightarrow \ \mathcal{Z}^{r}(E, \bullet)^{\ast}\big\}[-1].$$

Then, we have an exact sequence of Chow cohomology groups:
\[ \begin{tikzcd}
	\cdots \ar{r} & {\rm CHC}^{r}(X, m) \ar{r} & {\rm CH}^{r}(\widetilde{X},m) \oplus {\rm CHC}^{r}(X_{\rm sing}, m)\ar{r} & {\rm CHC}^{r}(E, m) \ar{r} & \cdots 
\end{tikzcd} \]

Let us recall that when $X$ is a smooth, the Bloch's higher Chow groups vanish in negative degrees. In codimension $r=1$, we have \cite{BLO}:
\begin{displaymath}
	{\rm CH}^{1}(X,m)= \left\{ \begin{array}{lll}
		{\rm Pic}(X), & \quad m=0 \\
		\Gamma(X, \mathcal{O}^{\ast}_{X}), & \quad m=1 \\
		0, & \quad m > 1.
	\end{array} \right. 
\end{displaymath} 

This fact implies, for any variety $X$, that ${\rm CHC}^{1}(X,m)=0$ for $m > 1$ (Proposition \ref{codim}). This group generally does not vanish in negative degrees. In this work we will focus on cohomological Chow groups in the case of codimension one. In particular we will review what happens with complex projective varieties $X$ of dimension $n \geq 3$ (with special emphasis on $n=3$), with singular locus $X_{\rm sing}$ of dimension zero (isolated singularities). The main results we prove in this work are the following:

\begin{theo}{\rm (=Theorem \ref{te})}
Let $X$ be an irreducible projective variety of dimension 3 over $\mathbb{C}$ with isolated singularities, and with dual complex associated to the normal crossing divisor $E$ such that $H^{2}(\Gamma)=0$. Then ${\rm CHC}^{1}(X, m)=0$ for $m \neq -2,-1, 0, 1$, ${\rm CHC}^{1}(X, 1)=\mathbb{C}^{\ast}$, and there is an exact sequence
\[ \begin{tikzcd}
0 \ar{r} & {\rm CHC}^{1}(X) \ar{r} & {\rm CH}^{1}(\widetilde{X}) \ar{r} & {\rm CHC}^{1}(E) \ar{r} & {\rm CHC}^{1}(X,-1) \ar{r} & 0 
\end{tikzcd} \]
\end{theo}

In higher dimensions we have:

\begin{theo}{\rm (=Proposition \ref{tee})}
	Let $E$ be a simple normal crossing divisor on a normal variety $X$ of dimension $d$, with associated dual complex $\Gamma(E)$ contractible. Then:
	\begin{displaymath}
		{\rm CHC}^{1}(E,m)= \left\{ \begin{array}{lll}
			\mathbb{C}^{\ast}, & \quad {\rm if} \ m=1 \\
			H^{-m}({\rm CH}^{1}(E^{[\bullet]})), & \quad {\rm if} \ -d+2 \leq m \leq 0 \\
			0, & \quad {\rm otherwise}.
		\end{array} \right. 
	\end{displaymath}
\end{theo}

\begin{theo}{\rm (=Theorem \ref{gen})}
Let $X$ be a complex projective variety of dimension $d$, with isolated singularities $X_{\rm sing}$. Let $p\colon \widetilde{X} \rightarrow X$ be a resolution of singularities, such that the associated dual complex $\Gamma(E)$ is contractible. Then ${\rm CHC}^{1}(X,m) \neq 0$ for $m=1, 0, -1, \ldots, d-1$, ${\rm CHC}^{1}(X, 1)=\mathbb{C}^{\ast}$, and there is an exact sequence:
\[ \begin{tikzcd}
	0 \ar{r} & {\rm CHC}^{1}(X) \ar{r} & {\rm CH}^{1}(\widetilde{X}) \ar{r} & {\rm CHC}^{1}(E) \ar{r} & {\rm CHC}^{1}(X, -1) \ar{r} & 0 
\end{tikzcd} \] 
\end{theo}

\section{Preliminaries}
\noindent The Chow cohomology groups (or motivic cohomology) were introduced in \cite{Han} using cubical hyperresolutions \cite{GNPP}, and higher cycle complexes \cite{BLO}. This is a cohomological generalization of the higher Chow groups to singular varieties. In this section we will briefly outline the construction of the motivic cohomology.

\subsection{Construction of semi-simplicial hyperresolutions}
In this work, we will compute the motivic cohomology groups of complete singular varieties with isolated singularities. The main idea is to express this cohomology in terms of cohomology groups of smooth projective varieties. The classical method is the one designed by Deligne in Hodge III \cite{DIII}, and uses what are known as {\it simplicial resolutions}. In this work, we will use {\it semi-simplicial resolutions}. Recall that a {\bf semi-simplicial scheme} $X_{\bullet}$ is a family of schemes $\{X_{p}\}_{p \in \mathbb{Z}_{\geq -1}}$ together with morphisms $d_{p,i}\colon X_{p} \rightarrow X_{p-1}$ for $0 \leq i \leq p$, called face maps
$$X_\bullet = \Bigl\{\begin{tikzcd}[column sep=2em]
	\cdots \ar[shift left=0.9em]{r}\ar[shift left=0.3em]{r}\ar[shift right=0.3em]{r}\ar[shift right=0.9em]{r} & X_2 \ar[shift left=0.6em]{r}{d_{2,0}}\ar{r}[description]{d}\ar[shift right=0.6em]{r}[swap]{d_{2,2}} & X_1 \ar[shift left=0.3em]{r}{d_{1,0}}\ar[shift right=0.3em]{r}[swap]{d_{1,1}} & X_0 \ar{r}{d_{0,0}} & X_{-1}
\end{tikzcd} \Bigr\}$$ such that $d_{p,i} \circ d_{p+1,i}=d_{p,i} \circ d_{p+1,i+1}$ with $0 \leq i \leq p$. This is semi-simplicial because it only considers the face morphisms, here the degeneracies are irrelevant. By composition of any sequence of $d_{j, i}$'s, for $j=0, 1, \ldots, p$, there is a unique morphism:
$$a_{p}\colon X_{p} \rightarrow X_{-1}.$$  

An alternative to Deligne's method of constructing simplicial resolutions, is the use of {\it cubical hyperresolutions} of Navarro-Aznar \cite{GNPP}, whre the main techinical advantage is that resulting semi-simplicial scheme is finite, bounded by the dimension of the original singular variety. Cubical hyperresolutions are obtained by a recursive process of Hironaka's resolution of singularities (this being the basis for the recurrence) \cite{Hir}. Let $X$ be a singular quasi-projective variety over a field $k$ that admits resolution of singularities, with singular locus $X_{\rm sing}$. Then, there is a diagram of the form:
\[ \begin{tikzcd}
	E=p^{-1}(X_{\rm sing}) \ar{r}{j} \ar{d}[swap]{q} & \widetilde{X} \ar{d}{p} \\ X_{\rm sing} \ar{r}[swap]{i} & X
\end{tikzcd} \]
where $\widetilde{X}$ is smooth, $p$ birational and proper morphism, and $E$ is a normal crossing divisor. This first step $X^{1}_{\bullet} \rightarrow X$ is also known as a {\bf 2-resolution} in \cite{GNPP} (abstract blow-up in \cite{MVW}), and $X_{\rm sing}$, $E$ are of lower dimension than $X$. Furthermore, this diagram induces a first step of a semi-simplicial resolution associated to $X$:
\[ \begin{tikzcd}[column sep=2em]
X_{1}=E \ar[shift left=0.3em]{r}{j} \ar[shift right=0.3em]{r}[swap]{q} & X_{0}=\widetilde{X} \sqcup X_{\rm sing} \ar{r}{(p,i)} & X.
\end{tikzcd} \] 

If $E$ and $X_{\rm sing}$ are smooth, this diagram defines a cubical hyperresolution of $X$. Since $E$ and $X_{\rm sing}$ may be singular in general, but both are of dimension less than $X$, we must follow the recurrence and replace the varieties $E$ and $X_{\rm sing}$ (under a reduction process) with diagrams of smooth varieties of lower dimension $X^{2}_{\bullet} \rightarrow X^{1}_{\bullet} \rightarrow X$ (applying resolution of singularities or 2-resolutions of a compatible form). This recurrence 
$$X^{r}_{\bullet} \ \rightarrow \ \cdots \ \rightarrow \ X^{2}_{\bullet} \ \rightarrow \ X^{1}_{\bullet} \ \rightarrow \ X,$$ where $X^{n+1}_{\bullet}$ is a 2-resolution of $X^{n}_{\bullet}$ provides a way to conctruct a cubical hyperresolution of $X$, details can be found in \cite{GNPP}. A cubical hyperresolution $X_{\bullet}={\rm red}(X^{1}_{\bullet}, X^{2}_{\bullet}, \ldots, X^{r}_{\bullet}) \rightarrow X$  defined under reductions (red) in the sense of \cite[Definition I.2.12]{GNPP} naturally defines a semi-simplicial scheme \cite[2.1.6]{Gui87}, where the $p$-th component is 
$$X_{p}=\coprod_{|\alpha|=p+1} X_{\alpha},$$ and the face morphisms are those induced by the cubical hyperresolution       
$$X_\bullet = \Bigl\{\begin{tikzcd}[column sep=2em]
	\cdots \ar[shift left=0.9em]{r}\ar[shift left=0.3em]{r}\ar[shift right=0.3em]{r}\ar[shift right=0.9em]{r} & X_2 \ar[shift left=0.6em]{r}{d_0}\ar{r}[description]{d_1}\ar[shift right=0.6em]{r}[swap]{d_2} & X_1 \ar[shift left=0.3em]{r}{d_0}\ar[shift right=0.3em]{r}[swap]{d_1} & X_0
\end{tikzcd}\Bigr\} \xrightarrow{a} X.$$ 

For this, we have a more appropriate definition:

\begin{defn}{\rm \cite[9.1]{Nava}}\label{Nava}
{\em Let X be a complex algebraic variety. A {\bf semi-simplicial hyperresolution} of $X$ is a semi-simplicial scheme $a\colon X_{\bullet} \rightarrow X$ augmented over $X$, such that
	:
\begin{enumerate}
	\item[$-$] $X_{p}$ are smooth quasi-projective varieties,
	\item[$-$] the morphisms $a_{p}\colon X_{p} \rightarrow X$ (given by composition) are proper, and
	\item[$-$] $a\colon X_{\bullet} \rightarrow X$ has the cohomological descent property, i.e., for any abelian sheaf $\mathcal{F}$ on $X$, the morphism $\mathcal{F} \rightarrow R a_{\ast} a^{\ast} \mathcal{F}$ is a quasi-isomorphism.
\end{enumerate}  
}
\end{defn}

The {\bf length} of this semi-simplicial hyperresolution is defined as the maximum $p$ such that $X_{p} \neq \emptyset$.

\begin{pro}\label{nor}
Let $X$ be a projective variety over $\mathbb{C}$ of dimension $d$ with isolated singularities $X_{\rm sing}$. There exists a semi-simplicial hyperresolution $X_{\bullet} \rightarrow X$ such that ${\rm dim} \ X_{p} \leq d - p$.
\end{pro}
\begin{proof}
The first step in the construction is to consider a resolution of singularities $p\colon \widetilde{X} \rightarrow X$, such that $p^{-1}(X_{\rm sing})=E=\bigcup^{k}_{i=1} E_{i}$ is a normal crossing divisor with smooth components, and smooth intersections of all orders. This induces a conmmutative square, with associated semisimplicial scheme
\[ \begin{tikzcd}[column sep=2em]
E \ar[shift left=0.3em]{r} \ar[shift right=0.3em]{r} & X_{0}=\widetilde{X} \sqcup X_{\rm sing} \ar{r} & X
\end{tikzcd} \] and with the property of cohomological descent. In our case $X_{\rm sing}$ is smooth, and the process continues resolving $E$. A semi-simplicial hyperresolution of $E$ is given by taking $$E^{[t]}:=\coprod_{|I|=t+1} E_{I}=\coprod \big(E_{i_{0}} \cap \cdots \cap E_{i_{t}}),$$ the face morphisms $d_{t, j}\colon E^{[t]} \rightarrow E^{[t-1]}$ are induced by inclusions
$$E_{i_{0}} \cap E_{i_{1}} \cap \cdots \cap E_{i_{t}} \ \hookrightarrow \ E_{i_{0}} \cap E_{i_{1}} \cap \cdots \cap \widehat{E_{i_{j}}} \cap \cdots \cap E_{i_{t}}$$ with $0 \leq j \leq t$, and this gives us a semi-simplicial scheme $E^{[\bullet]} \rightarrow E$ augmented over $E$:
\[ \begin{tikzcd}[column sep=2em]
	\cdots \ar[shift left=0.9em]{r}\ar[shift left=0.3em]{r}\ar[shift right=0.3em]{r}\ar[shift right=0.9em]{r} & E^{[2]} \ar[shift left=0.6em]{r} \ar{r} \ar[shift right=0.6em]{r} & E^{[1]} \ar[shift left=0.3em]{r} \ar[shift right=0.3em]{r} & E^{[0]} \ar{r} & E.
\end{tikzcd} \] with cohomological descent. This process defines a complete semi-simplicial hyperresolution of $X$:
\[ \begin{tikzcd}[column sep=2em]
\cdots \ar[shift left=0.6em]{r} \ar{r} \ar[shift right=0.6em]{r} & E^{[1]} \ar[shift left=0.3em]{r} \ar[shift right=0.3em]{r} & E^{[0]} \ar{d} \\ 
& & E \ar[shift left=0.3em]{r} \ar[shift right=0.3em]{r} & \widetilde{X} \sqcup X_{\rm sing} \ar{r} & X \\ \cdots \ar[shift left=0.9em]{r}\ar[shift left=0.3em]{r}\ar[shift right=0.3em]{r}\ar[shift right=0.9em]{r} & X_{2}=E^{[1]} \ar[shift left=0.6em]{r} \ar{r} \ar[shift right=0.6em]{r} & X_{1}=E^{[0]} \ar[shift left=0.3em]{r} \ar[shift right=0.3em]{r} & X_{0}=\widetilde{X} \sqcup X_{\rm sing} \ar{r} & X
\end{tikzcd} \]

\begin{tikzpicture}[overlay, remember picture]
\draw[thick] (2.5, 1.70) -- (13.2, 1.70); 
\end{tikzpicture} In this case ${\rm dim} \ X_{t}={\rm dim} \ E^{[t-1]}= \emptyset$ if $t > d$, and in general ${\rm dim} \ X_{t}={\rm dim} \ E^{[t-1]} \leq d -t$.
\end{proof}

Using \cite[I.2.15; I.6.9]{GNPP} it can be proved that every complex algebraic variety $X$ admits a semi-simplicial hyperresolution with cohomological descent, such that ${\rm dim} \ X_{p} \leq {\rm dim} \ X - p$, the length is bounded by the dimension of $X$. In \cite{Car}, Carlson  obtains results very similar to those obtained by the Barcelona school of Navarro-Aznar, using a polyhedral version. The most direct application of cubical hyperresolutions is the extension of cohomological functors first defined on the category of smooth varieties to the general category of arbitrary varieties \cite{GN}.

\subsection{Hanamura's cohomological Chow groups} 
As an application of the extension criterion of Guillén-Navarro \cite{GN}, we have the following construction. Let $X$ be a quasi-projective variety over a field $k$ that admits resolution of singularities, and consider $a\colon X_{\bullet} \rightarrow X$ its semi-simplicial hyperresolution. For each $X_{p}$ in the semi-simplicial scheme, take its Bloch's cycle complex $\mathcal{Z}^{r}(X_{p}, \bullet)$ as in \cite{BLO}, and form a fourth quadrant double complex:
$$\mathcal{Z}^{r}(X_{\bullet})^{\ast}:=\big[\mathcal{Z}^{r}(X_{0}, \bullet) \stackrel{d^{\ast}}{\longrightarrow} \mathcal{Z}^{r}(X_{1}, \bullet) \stackrel{d^{\ast}}{\longrightarrow} \cdots \stackrel{d^{\ast}}{\longrightarrow} \mathcal{Z}^{r}(X_{p}, \bullet) \stackrel{d^{\ast}}{\longrightarrow} \cdots \big]$$ whose vertical differentials $\partial_{B} \colon \mathcal{Z}^{r}(X_{p}, -q) \rightarrow \mathcal{Z}^{r}(X_{p}, -q-1)$ come from Bloch's cycle complexes; and the horizontal differentials are given by $d^{\ast}=\sum (-1)^{i}d^{\ast}_{i} \colon \mathcal{Z}^{r}(X_{p}, -q) \rightarrow \mathcal{Z}^{r}(X_{p+1}, -q)$, the alternating sums of the pull-backs. Then, this double complex can be seen as:
$$\mathcal{Z}^{p,q}_{0}(r):=\mathcal{Z}^{r}(X_{p}, -q); \quad p \geq 0, \quad q \leq 0.$$
\[ \begin{tikzcd}
	\mathcal{Z}^{r}(X_{0}) \ar{r}{d^{\ast}} & \mathcal{Z}^{r}(X_{1}) \ar{r}{d^{\ast}} & \mathcal{Z}^{r}(X_{2}) \ar{r} & ~ \\
	\mathcal{Z}^{r}(X_{0},1) \ar{u}{\partial} \ar{r}{d^{\ast}} & \mathcal{Z}^{r}(X_{1},1) \ar{r}{d^{\ast}} \ar{u}{\partial} & \mathcal{Z}^{r}(X_{2},1) \ar{u}{\partial} \ar{r} & ~ \\
	\mathcal{Z}^{r}(X_{0},2) \ar{u}{\partial} \ar{r}{d^{\ast}} & \mathcal{Z}^{r}(X_{1},2) \ar{u}{\partial} \ar{r}{d^{\ast}} & \mathcal{Z}^{r}(X_{2},2) \ar{u}{\partial} \ar{r} & ~ \\
	\ar{u} & \ar{u} & \ar{u}
\end{tikzcd} \]
We denote by $\mathcal{Z}^{r}(X, \bullet)^{\ast}$ its total complex 
$$\mathcal{Z}^{r}(X, \bullet)^{\ast}:={\rm Tot}\big(\mathcal{Z}^{r}(X_{\bullet})^{\ast}\big)={\rm Tot}\big(\mathcal{Z}^{r}(X_{0}, \bullet) \stackrel{d^{\ast}}{\longrightarrow} \mathcal{Z}^{r}(X_{1}, \bullet) \stackrel{d^{\ast}}{\longrightarrow} \cdots \stackrel{d^{\ast}}{\longrightarrow} \mathcal{Z}^{r}(X_{p}, \bullet) \stackrel{d^{\ast}}{\longrightarrow} \cdots \big)$$
and we call it the {\bf cohomological cycle complex} of $X$. 

\begin{re} 
{\em In Hanamura's construction, one chooses appropiate distinguished subcomplexes, so that the pull-back $d^{\ast}$ is well-defined \cite{Han}. The main result on the cohomological cycle complex is that it is independent of the hyperresolution up to isomorphisms in the derived category \cite[Theorem I]{Han}.}
\end{re}

\begin{defn}\cite[Definition 2.4]{Han}
{\em For a quasi-projective variety $X$, we define
$${\rm CHC}^{r}(X,m):=H^{-m}(\mathcal{Z}^{r}(X_{\bullet})^{\ast})$$ the {\bf cohomological Chow groups} (or motivic cohomology) of $X$. }
\end{defn}

By the extension criterion of Guillén-Navarro \cite[Theorem 2.1.5]{GN}, many properties that are satisfied for smooth varieties can be translated to any variety, using resolution of singularities recursively. The higher Chow groups functor ${\rm CH}^{r}(-,m) \colon {\bf Sm}(k) \rightarrow {\bf Ab}$ admits a natural extension to the contravariant functor:
$${\rm CHC}^{r}(-,m) \colon {\bf QuProj}(k) \ \rightarrow \ {\bf Ab}.$$ Such that if $X$ is smooth, then ${\rm CHC}^{r}(X,m)={\rm CH}^{r}(X,m)$. Moreover, there is a fourth quadrant spectral sequence
$$E^{p,q}_{1}(r):=H^{q}\big(\mathcal{Z}^{r}(X_{p}, \bullet) \big)={\rm CH}^{r}(X_{p}, -q) \Rightarrow H^{p+q}\big( {\rm Tot}\big(\mathcal{Z}^{r}(X_{\bullet}, \bullet)) \big)={\rm CHC}^{r}(X, -p-q)$$ where $X_{p}=\coprod_{|\alpha|=p+1} X_{\alpha}$. If $X$ is a projective singular variety over $\mathbb{C}$ with singular locus $X_{\rm sing}$, consider a resolution of singularities $p \colon \widetilde{X} \rightarrow X$ such that $E=p^{-1}(X_{\rm sing})$ is a normal crossing divisor. With the above notations and the notations of Proposition \ref{nor}, there is a fourth quadrant spectral sequence associated to the semi-simplicial hyperresolution $E^{[\bullet]} \rightarrow E$ \cite[Proposition 5.1]{Navarro}:
$$E^{a,b}_{1}(r):={\rm CH}^{r}(E^{[a]}, -b) \ \Rightarrow \ {\rm CHC}^{r}(E, -a-b); \qquad a \geq 0, \quad b \leq 0.$$

\begin{re}
{\em Note the slight difference in the notation, for Hanamura's cohomological Chow groups we use ${\rm CHC}^{\ast}(-,m)$ and for Bloch's higher Chow groups ${\rm CH}^{\ast}(-,m)$. This is the original notation of \cite{Han}.
}
\end{re}


\section{Computations: codimension one}
\noindent Let $X$ be a complex quasi-projective algebraic variety of dimension $d$, with singular locus $X_{\rm sing}$. Consider a commutative square of resolution of singularities:
\[ \begin{tikzcd} 
E=p^{-1}(X_{\rm sing}) \ar{r} \ar{d} & \widetilde{X} \ar{d} \\ X_{\rm sing} \ar{r} & X.
\end{tikzcd} \] 

The cohomological cycle complex associated to $X$ is
\begin{equation}\label{cycom}
\mathcal{Z}^{r}(X, \bullet)^{\ast}={\rm Cone}\big\{\mathcal{Z}^{r}(\widetilde{X}, \bullet) \oplus \mathcal{Z}^{r}(X_{\rm sing}, \bullet)^{\ast} \rightarrow \mathcal{Z}^{r}(E, \bullet)^{\ast}\big\}[-1]
\end{equation}

If $X_{\rm sing}$ and $E$ are smooth, we have a long exact sequence of the form
\[ \begin{tikzcd}
	\cdots \ar{r} & {\rm CH}^{r}(E,m+1)\ar{r} & {\rm CHC}^{r}(X,m) \ar{r} &
	{\rm CH}^{r}(\widetilde{X},m) \oplus {\rm CH}^{r}(X_{\rm sing},m) \ar{r} & \cdots
\end{tikzcd} \] associated to above square. If $E$ and $X_{\rm sing}$ are not smooth, the above diagram can be extended to a cubical hyperresolution $X_{\bullet} \rightarrow X$ of the form
\[ \begin{tikzcd}
E_{\bullet}	\ar{r} \ar{d} & E \ar{r} \ar{d} & \widetilde{X} \ar{d} \\ X_{\rm sing \bullet} \ar{r} & X_{\rm sing} \ar{r} & X
\end{tikzcd} \] where $(E_{\bullet}	\rightarrow X_{\rm sing \bullet})$ is a hyperresolution of $(E	\rightarrow X_{\rm sing})$. The hyperresolutions $E_{\bullet}$ and $X_{\rm sing \bullet}$ are strictly smaller than the length of $X_{\bullet}$, and the complex (\ref{cycom}) induces an exact sequence of Mayer-Vietoris type on the motivic cohomology groups:
\[ \begin{tikzcd}
	\cdots \ar{r} & {\rm CHC}^{r}(X,m)\ar{r} & {\rm CH}^{r}(\widetilde{X},m) \oplus {\rm CHC}^{r}(X_{\rm sing},m) \ar{r} & {\rm CHC}^{r}(E,m) \ar{r} & \cdots
\end{tikzcd} \]  

Here $\widetilde{X}$ is smooth, then ${\rm CHC}^{r}(X,m)={\rm CH}^{r}(X,m)$ and these groups vanish in negative degrees for codimension reasons. If $X$ is a normal variety of dimension $d$, then ${\rm codim}(X_{\rm sing}) \geq 2$, and 
$${\rm CHC}^{r}(X_{\rm sing}, m)=0, \quad {\rm if} \quad -m > d  - {\rm codim}(X_{\rm sing}) -r={\rm dim}(X_{\rm sing}) -r$$ by Proposition \ref{dim}. This construction suggests that the motivic cohomology of $X$ can be computed (recursively) in terms of the smooth components that appear in the hyperresolution. Explicitly, in terms of a smooth variety $\widetilde{X}$, and the hyperresolutions of the singular varieties $X_{\rm sing}$ and $E$. Recall that in codimension one, for $X$ smooth and irreducible, S. Bloch \cite{BLO} shows that: 
\begin{equation} \label{eqn:BLO} 
	{\rm CH}^{1}(X,m) \cong H^{2-m}_{\mathcal{M}}(X, \mathbb{Z}(1)) = \left\{ \begin{array}{lll}
		{\rm Pic}(X), & \quad m=0 \\
		\Gamma(X, \mathcal{O}^{\ast}_{X}), & \quad m=1 \\
		0, & \quad m > 1.
	\end{array} \right. 
\end{equation} 

In the singular case we have the following result:

\begin{pro}\label{codim}
${\rm CHC}^{1}(X,m)=0$ for $m \geq 2$.
\end{pro} 
\begin{proof}
Let $X_{\bullet}  \rightarrow X$ be a cubical hyperresolution of $X$. The result is obtained by applying induction over the length of the hyperresolution. For $r=1$, consider a 2-resolution
\[ \begin{tikzcd} 
E \ar{r} \ar{d} & \widetilde{X} \ar{d} \\ X_{\rm sing} \ar{r} & X.
\end{tikzcd} \] with smooth components. Then, we have an exact sequence of cohomological Chow groups:
\[ \begin{tikzcd}
\cdots \ar{r} & {\rm CHC}^{1}(X, 2) \ar{r} & {\rm CH}^{1}(\widetilde{X},2) \oplus {\rm CH}^{r}(X_{\rm sing}, 2)\ar{r} & {\rm CH}^{1}(E, 2) \ar{r} & {\rm CHC}^{1}(X, 1) \ar{r} & \cdots 
\end{tikzcd} \]

By (\ref{eqn:BLO}), ${\rm CH}^{1}(Y,m)=0$ for $m \geq 2$ if $Y$ is smooth.
\end{proof}

In general, ${\rm CHC}^{1}(X,m)$ can be non-zero in negative degrees. The cohomological Chow groups in codimension 1 for a curve are well known, and can be seen in \cite[Proposition 1.1, 1.2]{Hans}. For a normal surface it can also be seen in \cite[Proposition 2.1, 2.2]{Hans}. Here we will make a sketch to illustrate some properties. Recall that the singularities of a normal surface are isolated. 

\begin{ex}
{\em Let $S$ be a projective surface with singular locus $S_{\rm sing}=\{s\}$ an isolated singularity. The 2-resolution of $S$ is given by the diagram: }
\end{ex}
\[\begin{tikzcd}
	p^{-1}(s)=E \ar{r} \ar{d} & \widetilde{S} \ar{d}{p} \\ \{s\} \ar{r} & S
\end{tikzcd}\] where $E= \bigcup^{N}_{i=1} E_{i}$ is a simple normal crossing divisor, with irreducible and projective components. The cohomological cycle complex associated to $S$ is given by
$$\mathcal{Z}^{r}(S,\bullet)^{\ast}={\rm Cone}\big\{\mathcal{Z}^{r}(\widetilde{S}, \bullet) \oplus \mathcal{Z}^{r}(\{s\}, \bullet) \rightarrow \mathcal{Z}^{r}(E_{\bullet})^{\ast}\big\}[-1].$$ 

In this case, the process reduces the study of motivic cohomology of $S$ to the normal crossing divisor $E$. Let $\{E_{i}\}^{N}_{i=1}$ be the irreducible components of $E$, and $E_{ij}=E_{i} \cap E_{j}$ for $i < j$. Let $E^{[0]}:=\coprod E_{i}$ be the resolution of singularities of $E$, and $E^{[1]}:=\coprod_{i < j} E_{ij}$. The face morphisms are given as in Example \ref{nor}. The cubical hyperresolution of $S$ is given by:
\[ \begin{tikzcd}[row sep=0.3em, column sep=1em]
	& E^{[1]} \sqcup E^{[1]} \ar{rr}\ar[]{dd}\ar{dl} && E^{[0]} \ar{dd}\ar{dl}\ar{ddrrr} \\
	\{s\} \ar[crossing over]{rr}\ar{dd} && \{s\} \ar{ddrrr} \\
	& E^{[1]} \ar{rr}\ar{dl} && E \ar[equals]{r}\ar{dl} & E \ar{dl}\ar{rr} && \widetilde{S} \ar{dl} \\
	\{s\} \ar{rr} && \{s\} \ar[equals]{r}\ar[leftarrow, crossing over]{uu} & \{s\} \ar{rr} && S
\end{tikzcd} \quad \stackrel{\rm reduction}{\longrightarrow} \quad \begin{tikzcd}[row sep=0.3em, column sep=1em]
	& E^{[1]} \sqcup E^{[1]} \ar{rr}\ar[]{dd}\ar{dl} && E^{[0]} \ar{dd}\ar{dl} \\
	\{s\} \ar[crossing over]{rr}\ar{dd} && \{s\} \\
	& E^{[1]} \ar{rr}\ar{dl} && \widetilde{S} \ar{dl} && \\
	\{s\} \ar{rr} && S \ar[leftarrow, crossing over]{uu} &&
\end{tikzcd} \] with associated semi-simplicial hyperresolution
\[ \begin{tikzcd}[column sep=2em]
	S_{2}=E^{[1]} \sqcup E^{[1]} \ar[shift left=0.6em]{r} \ar{r} \ar[shift right=0.6em]{r} & S_{1}=E^{[1]} \sqcup E^{[0]} \sqcup \{s\} \ar[shift left=0.3em]{r} \ar[shift right=0.3em]{r} & S_{0}=\widetilde{S} \sqcup \{s\} \sqcup \{s\}\ar{r} & S.
\end{tikzcd} \]  Although a more efficient way to consider this hyperresolution is to omit some terms: 
\[ \begin{tikzcd}[column sep=2em]
	E^{[1]} \ar[shift left=0.3em]{r} \ar[shift right=0.3em]{r} & E^{[0]} \ar{d} \\ 
	& E \ar[shift left=0.3em]{r} \ar[shift right=0.3em]{r} & \widetilde{S} \sqcup \{s\} \ar{r} & S \\ S_{2}=E^{[1]} \ar[shift left=0.6em]{r} \ar{r} \ar[shift right=0.6em]{r} & S_{1}=E^{[0]} \ar[shift left=0.3em]{r} \ar[shift right=0.3em]{r} & S_{0}=\widetilde{S} \sqcup \{s\} \ar{r} & S
\end{tikzcd} \]

\begin{tikzpicture}[overlay, remember picture]
	\draw[thick] (3.3, 1.35) -- (12.2, 1.35); 
\end{tikzpicture} 
it can be proved that the resulting cohomological cycle complexes are quasi-isomorphic. The cohomological cycle complex associated to $E$ is
$$\mathcal{Z}^{r}(E, \bullet)^{\ast}={\rm Cone}^{\bullet}\big\{\mathcal{Z}^{r}(E^{[0]}, \bullet) \stackrel{d^{\ast}}{\longrightarrow} \mathcal{Z}^{r}(E^{[1]}, \bullet)\big\}[-1].$$ 

The {\bf dual graph} $\Gamma(E)$ associated to $E$ is the graph with vertices corresponding to the components $E_{i}$, and the edges corresponding to $E_{ij}$. The resulting graph $\Gamma$ is connected since $E$ is connected. The chain complex associated to $\Gamma(E)$ is a homological complex of two terms: 
$$C_{\bullet}(\Gamma): C_{1}(\Gamma) \stackrel{\partial}{\longrightarrow} C_{0}(\Gamma),$$ where $C_{0}(\Gamma)=\mathbb{Z}[E_{i}]$, $C_{1}(\Gamma)=\mathbb{Z}[E_{ij}]$, and the differential sends $E_{ij}$ to $E_{j} - E_{i}$. The homology is easy to describe, $H_{0}(\Gamma)=\mathbb{Z}$ because $\Gamma$ is connected, and $H_{1}(\Gamma)=\mathbb{Z}^{e-v+1}$ (where $e=\#$edges and $v=\#$vertices, in $\Gamma$). In particular, if $\Gamma$ is a tree, then $H_{1}(\Gamma)=0$. The cohomological chain complex $C^{\bullet}(\Gamma)$ is the complex $\delta \colon C^{0}(\Gamma) \rightarrow C^{1}(\Gamma)$, where $C^{i}(\Gamma)={\rm Hom}(C_{i}(\Gamma), \mathbb{Z})$ and $\delta$ is the dual of $\partial$. Explicitly, $C^{0}(\Gamma)=\mathbb{Z}[e_{i}]$ (resp. $C^{1}(\Gamma)=\mathbb{Z}[E_{ij}]$) is the free $\mathbb{Z}$-module generated by $\{e_{i}\}$ dual to $\{E_{i}\}$ (resp. generated by $\{e_{ij}\}$ dual to $\{E_{ij}\}$), and 
$$\delta(e_{i})=\sum_{i < j}e_{ij} - \sum_{l < i}e_{li}.$$ 

Since ${\rm CH}^{1}(E^{[0]},1)=C^{0}(\Gamma) \otimes \mathbb{C}^{\ast}$ and ${\rm CH}^{1}(E^{[1]},1)=C^{1}(\Gamma) \otimes \mathbb{C}^{\ast}$, then $\delta \otimes 1\colon C^{0}(\Gamma) \otimes \mathbb{C}^{\ast} \rightarrow C^{1}(\Gamma) \otimes \mathbb{C}^{\ast}$ is identified with $d^{\ast}\colon {\rm CH}^{1}(E^{[0]},1) \rightarrow{\rm CH}^{1}(E^{[1]},1)$. Recall that ${\rm CHC}^{1}(E, n)=H^{-n}\big(\mathcal{Z}^{1}(E, \bullet)^{\ast}\big)$, and we have a long exact sequence of the form:
\[\begin{tikzcd}
0 \ar{r} & {\rm CHC}^{1}(E,1) \ar{r} & {\rm CH}^{1}(E^{[0]},1) \ar{r}{d^{\ast}} & {\rm CH}^{1}(E^{[1]},1) \ar{r} & {\rm CHC}^{1}(E) \ar{r} & {\rm CH}^{1}(E^{[0]}) \ar{r} & 0.
\end{tikzcd} \] 

By \cite[Proposition 1.3]{Hans}, ${\rm CHC}^{1}(E, m)=0$ for $m \neq 0, 1$ and ${\rm CHC}^{1}(E, 1)= \mathbb{C}^{\ast}$, and there exist an exact sequence
\[\begin{tikzcd}
0 \ar{r} & H^{1}(\Gamma) \otimes \mathbb{C}^{\ast} \ar{r} & {\rm CHC}^{1}(E) \ar{r} & {\rm CH}^{1}(E^{[0]}) \ar{r} & 0
\end{tikzcd}\] 

By definition, the cohomological Chow group of $S$ in codimension one is: ${\rm CHC}^{1}(S, m)=H^{-m}(\mathcal{Z}^{1}(S,\bullet)^{\ast})$. This induces a long exact sequence
\[\begin{tikzcd}[row sep=0.6em]
	0 \ar{r} & {\rm CHC}^{1}(S,1) \ar{r} & {\rm CH}^{1}(\widetilde{S},1) \oplus \mathbb{C}^{\ast} \ar{r} & {\rm CHC}^{1}(E,1) \ar{r} & {\rm CHC}^{1}(S) \\ \ar{r} & {\rm CH}^{1}(\widetilde{S}) \ar{r} &{\rm CHC}^{1}(E) \ar{r} & {\rm CHC}^{1}(S, -1) \ar{r} & 0
\end{tikzcd} \]

By \cite[Proposition 2.1]{Hans}, if $S$ is an irreducible normal projective surface over $\mathbb{C}$. Then ${\rm CHC}^{1}(S, m)=0$ for $n \neq -1, 0, 1$, ${\rm CHC}^{1}(S,1)=\mathbb{C}^{\ast}$, and there is an exact sequence
	\[\begin{tikzcd}
		0 \ar{r} & {\rm CHC}^{1}(S) \ar{r} & {\rm CH}^{1}(\widetilde{S}) \ar{r} & {\rm CHC}^{1}(E) \ar{r} & {\rm CHC}^{1}(S, -1) \ar{r} & 0
	\end{tikzcd} \]

\subsection{Motivic cohomology of threefolds}
\noindent Let $X$ be a complex singular projective variety of dimension 3 (threefold), with singular locus $X_{\rm sing}$. Let $p\colon \widetilde{X} \rightarrow X$ be a desingularization of $X$ such that $E=p^{-1}(X_{\rm sing})$ is a normal crossing divisor. Let us start by assuming that $X_{\rm sing}$ is a connected smooth curve, and $E$ is a connected smooth surface. Then we have a semi-simplicial hyperresolution of the form:
$$\begin{tikzcd}[column sep=2em]
	X_{1}=E \ar[shift left=0.3em]{r} \ar[shift right=0.3em]{r} & X_{0}=\widetilde{X} \sqcup X_{\rm sing} \ar{r}{a} & X,
\end{tikzcd}$$ with cohomological cycle complex given by
$$\mathcal{Z}^{r}(X_{\bullet})^{\ast}={\rm Cone}\big\{\mathcal{Z}^{r}(\widetilde{X}, \bullet) \oplus \mathcal{Z}^{r}(X_{\rm sing}, \bullet) \rightarrow \mathcal{Z}^{r}(E, \bullet)\big\}[-1].$$ 

In codimension $r=1$, we have a long exact sequence
\[ \begin{tikzcd}[row sep=0.6em]
0 \ar{r} & {\rm CHC}^{1}(X,1) \ar{r} & \Gamma(\widetilde{X}, \mathcal{O}^{\ast}_{\widetilde{X}}) \oplus \Gamma(X_{\rm sing}, \mathcal{O}^{\ast}_{X_{\rm sing}}) \ar{r} & \Gamma(E, \mathcal{O}^{\ast}_{E}) \ar{r} & {\rm CHC}^{1}(X) \\ \ar{r} & {\rm CH}^{1}(\widetilde{X}) \oplus {\rm CH}^{1}(X_{\rm sing}) \ar{r} & {\rm CH}^{1}(E) \ar{r} & {\rm CHC}^{1}(X,-1) \ar{r} & 0.
\end{tikzcd} \]

Under the identification (\ref{eqn:BLO}), the above sequence gives us the following:
\[ \begin{tikzcd}[row sep=0.6em]
	0 \ar{r} & {\rm CHC}^{1}(X,1) \ar{r} & \mathbb{C}^{\ast} \oplus \mathbb{C}^{\ast} \ar{r} & \mathbb{C}^{\ast} \ar{r} & {\rm CHC}^{1}(X) \\
	\ar{r} & {\rm CH}^{1}(\widetilde{X}) \oplus {\rm CH}^{1}(X_{\rm sing}) \ar{r} & {\rm CH}^{1}(E) \ar{r} & {\rm CHC}^{1}(X, -1) \ar{r} & 0.
\end{tikzcd} \]

\begin{pro}
	With the above conditions, we have that ${\rm CHC}^{1}(X,m)=0$ for $m \neq -1, 0, 1$, and ${\rm CHC}^{1}(X,1)=\mathbb{C}^{\ast}$. Moreover 
	$${\rm CHC}^{1}(X)={\rm Ker}\{{\rm CH}^{1}(\widetilde{X}) \oplus {\rm CH}^{1}(X_{\rm sing}) \ \rightarrow \ {\rm CH}^{1}(E)\},$$ and 
	$${\rm CHC}^{1}(X,-1)={\rm Coker}\{{\rm CH}^{1}(\widetilde{X}) \oplus {\rm CH}^{1}(X_{\rm sing}) \ \rightarrow \ {\rm CH}^{1}(E)\}.$$
\end{pro}
\begin{proof}
It is immediate from the above exact sequence. 
\end{proof} 

\noindent Now suppose that $X_{\rm sing}=\{x\}$ is an isolated singularity, and consider a 2-resolution of $X$:
\[ \begin{tikzcd}
	E \ar{r} \ar{d} & \widetilde{X} \ar{d}{p} \\ \{x\} \ar{r} & X.
\end{tikzcd} \] 

This desingularization extends to a cubical hyperresolution by resolving $E \rightarrow X_{\rm sing}$, but in this case $X_{\rm sing}$ is non-singular, and the process continues via a 2-resolution of $E$. By definition, ${\rm CHC}^{r}(X,m)$ is the $m^{\rm th}$-homology of the complex
$$\mathcal{Z}^{r}(X_{\bullet})^{\ast}:={\rm Cone}\{\mathcal{Z}^{r}(\widetilde{X}, \bullet) \oplus \mathcal{Z}^{r}(\{x\}, \bullet) \rightarrow \mathcal{Z}^{r}(E_{\bullet})^{\ast}\}[-1].$$ Since $\widetilde{X}$ and $X_{\rm sing}$ are both smooth, then ${\rm CHC}^{1}(\widetilde{X}, m)={\rm CH}^{1}(\widetilde{X}, m)$  and ${\rm CHC}^{1}(X_{\rm sing}, m)={\rm CH}^{1}(X_{\rm sing}, m)$, and vanish in negative degrees. Recall by Proposition \ref{codim}, that ${\rm CHC}^{1}(X, m)=0$  for $m > 1$. Furthermore, we have the following result, for any type of singularities.

\begin{pro}\label{dim}
	Let $X$ be a projective variety over $\mathbb{C}$, of dimension $d$. Then 
	$${\rm CHC}^{r}(X ,m)=0$$ whenever $r > d + m$.
\end{pro}
\begin{proof}
Let $X_{\bullet} \rightarrow X$ be a cubical hyperresolution of lenght $r$, we do recursion on $r$ and this leads to the case $r=1$. In this case, let $X_{\rm sing}$ be the singular locus of $X$, and consider the following diagram given by resolution of singularities:
\[ \begin{tikzcd}
E= p^{-1}(X_{\rm sing}) \ar{r} \ar{d} & \widetilde{X} \ar{d}{p} \\ X_{\rm sing} \ar{r} & X.
\end{tikzcd} \] 

By assumption, the cubical hyperresolution is of length one, hence $X_{\rm sing}$ and $E$ are smooth. The associated semi-simplicial scheme is:
	\begin{equation} \label{eqn:hyperresolution} X_\bullet = \Bigl\{\begin{tikzcd}[column sep=2em]
			X_1 = E \ar[shift left=0.3em]{r} \ar[shift right=0.3em]{r} & X_0 = \widetilde{X} \sqcup X_{\mathrm{sing}} 
		\end{tikzcd}\Bigr\} \xrightarrow{a} X
	\end{equation}

This induces a long exact sequence in cohomological Chow groups
\[ \begin{tikzcd}
\cdots \ar{r} & {\rm CH}^{r}(X_{1}, m+1) \ar{r} & {\rm CHC}^{r}(X, m) \ar{r} & {\rm CH}^{r}(X_{0}, m) \ar{r} & {\rm CH}^{r}(X_{1}, m) \ar{r} & \cdots
\end{tikzcd} \] 

Since ${\rm dim}(X_1 = E) < d$, we see that ${\rm CH}^r(X_1,m+1) = 0$ by dimension considerations. Similarly, $X_{0}=\widetilde{X} \sqcup X_{\mathrm{sing}}$ does not contribute to ${\rm CHC}^r(X,m)$ in the range $r>m+d$, hence the claim. If $E$ and $X_{\rm sing}$ are not smooth, we take 2-solutions for both and conclude that ${\rm CHC}^{r}(E, m)=0$ and ${\rm CHC}^{r}(X_{\rm sing}, m)=0$. By induction we have the result.
\end{proof}    

A first hyperresolution associated to $X$ is:  \begin{tikzcd}[column sep=2em]
E \ar[shift left=0.3em]{r} \ar[shift right=0.3em]{r} & X_{0}=\widetilde{X} \sqcup X_{\rm sing} \ar{r} & X.
\end{tikzcd} Then, putting all these results together we see that ${\rm CHC}^{1}(E,m)=0$ if $m < -1$ and ${\rm CHC}^{1}X, m)=0$ for $m < -2$, so that we have a long exact sequence
\begin{equation}\label{eq1}
\begin{tikzcd}[row sep=0.6em]
0 \ar{r} & {\rm CHC}^{1}(X, 1) \ar{r} & {\rm CH}^{1}(X_{0}, 1) \ar{r} & {\rm CHC}^{1}(E, 1) \ar{r} & {\rm CHC}^{1}(X) \ar{r} & {\rm CH}^{1}(X_{0}) \\ \ar{r} & {\rm CHC}^{1}(E) \ar{r} & {\rm CHC}^{1}(X, -1) \ar{r} & 0 \ar{r} & {\rm CHC}^{1}(E, -1) \ar{r} & {\rm CHC}^{1}(X, -2) 
\end{tikzcd} 
\end{equation}

The first observation is that ${\rm CHC}^{1}(E, -1) \cong {\rm CHC}^{1}(X, -2)$. To compute the cohomological Chow groups of $X$ in codimension $1$ and degree $-2$, consider the following. Let $E=\bigcup^{N}_{i=1} E_{i}$ be a simple normal crossing divisor with each component an irreducible and projective surface. Let $E_{ij}=E_{i} \cap E_{j}$ be the intersection of any two irreducible components of $E$, so the intersections $E_{ij}$ are curves, and $E_{ijl}=E_{i} \cap E_{j} \cap E_{l}$ which are points. Set $E^{[1]}:=\coprod_{i < j} E_{ij}$, and $E^{[2]}:=\coprod_{i < j < l} E_{ijl}$. The resolution of singularities of $E$ is:
\[ \begin{tikzcd}
E^{(1)} \coprod E^{(1)} \ar{r} \ar{d} & E^{[0]} \ar{d} \\ E^{(1)} \ar{r} & E
\end{tikzcd} \] where $E^{(0)}:=E_{1} \cup \cdots \cup E_{N}=E$, and $E^{[0]}:=\coprod E_{i}$ is the canonical desingularization of $E^{(0)}$. Here $E^{(1)}$ is the union of 2-fold intersections of the components of $E$ which is the singular locus of $E$, the inverse image is identified with $E^{(1)} \coprod E^{(1)}$. In general $E^{(t)}:=\bigcup(E_{i_{0}} \cap \cdots \cap E_{i_{t}})$, and then $E^{[t]}$ is the canonical resolution of singularities of $E^{(t)}$. In our case, the cubical hyperresolution associated to $E$ is:
\[ \begin{tikzcd}[row sep=0.3em, column sep=1em]
	& E^{[2]} \coprod E^{[2]} \ar{rr}\ar[]{dd}\ar{dl} && E^{[1]} \coprod E^{[1]} \ar{dd}\ar{dl} \\  E^{[2]} \ar[crossing over]{rr}\ar{dd} && E^{[1]} \\
	& E^{[2]} \ar{rr}\ar{dl} && E^{[0]} \ar{dl} && \\
	E^{[2]} \ar{rr} && E \ar[leftarrow, crossing over]{uu} &&
\end{tikzcd} \] with semi-simplicial scheme augmented over $E$:
\[ \begin{tikzcd}[column sep=2em]
E^{[2]} \coprod E^{[2]} \ar[shift left=0.6em]{r} \ar{r} \ar[shift right=0.6em]{r} & \big(E^{[2]} \coprod E^{[2]}\big) \coprod \big(E^{[1]} \coprod E^{[1]}\big) \ar[shift left=0.3em]{r} \ar[shift right=0.3em]{r} & E^{[2]} \coprod E^{[1]} \coprod E^{[0]} \ar{r} & E
\end{tikzcd} \]

This semi-simplicial hyperresolution is redundant, a more efficient simplicial scheme is:
\[ \begin{tikzcd}[column sep=2em]
	E^{[2]}=\coprod (E_{i} \cap E_{j} \cap E_{l}) \ar[shift left=0.6em]{r} \ar{r} \ar[shift right=0.6em]{r} & E^{[1]}=\coprod (E_{i} \cap E_{j}) \ar[shift left=0.3em]{r} \ar[shift right=0.3em]{r} & E^{[0]}=\coprod E_{i} \ar{r} & E.
\end{tikzcd} \] 
 
Then the cohomological cycle complex associated to $E$ is:
$$\mathcal{Z}^{r}(E, \bullet)^{\ast}:={\rm Tot}^{\bullet}\big\{   
\mathcal{Z}^{r}(E^{[0]}, \bullet) \longrightarrow \mathcal{Z}^{r}(E^{[1]}, \bullet) \longrightarrow \mathcal{Z}^{r}(E^{[2]}, \bullet)\big\}.$$ 

By definition, the motivic cohomology of $E$ is ${\rm CHC}^{r}(E,m)=H^{-m}\big(\mathcal{Z}^{r}(E, \bullet)^{\ast}\big)$. To compute these groups, we will use the following general technique \cite{ABW, Ste, Step}: consider a resolution of singularities of $X$ 
\[ \begin{tikzcd} 
E=p^{-1}(X_{\rm sing}) \ar{r} \ar{d} & \widetilde{X} \ar{d}{p} \\ X_{\rm sing} \ar{r} & X
\end{tikzcd} \] where $E=\bigcup^{N}_{i=1}E_{i}$ is a simple normal crossing divisor, let us also assume that each intersection $E_{i_{0}} \cap \cdots \cap E_{i_{t}}$ is irreducible \cite{Ste}. Now we will associate to $E$ a simplicial complex, which describes the way in which the different components of $E$ intersect.

\begin{defn}\cite[Section 1]{Ste}
{\em Let $E=\bigcup^{N}_{i=1}E_{i}$ be a simple normal crossing divisor. The {\bf dual complex} $\Gamma(E)$ associated to $E$ is the simplicial complex constructed as follows: for $l \in \mathbb{Z}_{\geq 0}$, an $l$-dimensional cell corresponds to each irreducible component of $(l+1)$-fold intersections $E_{i_{0}} \cap \cdots \cap E_{i_{l}}$. }
\end{defn}

The 0-simplices correspond to the irreducible components $E_{i}$ of $E$. The 1-simplices (edges) correspond to the intersections $E_{ij}=E_{i} \cap E_{j}$. If $X$ is a surface, $\Gamma(E)$ is the dual graph associated to the resolution $p\colon \widetilde{X} \rightarrow X$. In this sense, the concept of dual complex generalizes the notion of dual graph. In general $\Gamma(E)$ is a CW-complex, and is a simplicial complex if only if all the intersections $E_{i_{0}} \cap \cdots \cap E_{i_{t}}$ are irreducible \cite[Section 1]{Ste}. For another resolution $p' \colon \widetilde{X}' \rightarrow X$, we have different dual complexes $\Gamma(E)$  and $\Gamma(E')$. But by \cite[Theorem 1]{Ste}, the resulting complexes $\Gamma(E)$  and $\Gamma(E')$ have the same homotopy type. We say that $\Gamma(E)$ has {\it irreducible intersections} if $E$ is a simple normal crossing divisor, such that each $E_{i_{0}} \cap \cdots \cap E_{i_{t}}$ is irreducible. Then, we have following technical lemma:

\begin{lem}\label{lem}
Suppose that $E$ is a simple normal crossing divisor with irreducible intersections, then the complex ${\rm CH}^{1}(E^{[\bullet]}, 1)$ coincides with the complex of chains $C^{\bullet}(\Gamma) \otimes \mathbb{C}^{\ast}$ on the dual complex $\Gamma(E)$.
\end{lem}
\begin{proof}
Denote by $E^{\bullet, -1}_{1}(1):={\rm CH}^{1}(E^{[\bullet]}, 1)$ the complex given by the higher Chow groups of codimension $1$ and degree $1$. Recall that if $Y$ is a smooth and projective variety over $\mathbb{C}$, then ${\rm CH}^{1}(Y, 1)=\Gamma(Y, \mathcal{O}^{\ast}_{Y}) \cong (\mathbb{C}^{\ast})^{\pi_{0}(Y)}$. Let $E$ be a normal crossing divisor with associated dual complex $\Gamma(E)$ with $E_{i_{0}} \cap \cdots \cap E_{i_{t}} \neq \emptyset$, and $E^{[t]}= \coprod_{i_{0} < \cdots < i_{t}} E_{i_{0}} \cap \cdots \cap E_{i_{t}}$. Then, the sheaf $\mathcal{O}^{\ast}_{E^{[t]}}$ admits a decomposition of the form $$\mathcal{O}^{\ast}_{E^{[t]}}=\bigoplus_{i_{0} < \cdots < i_{t}} \mathcal{O}^{\ast}_{E_{i_{0}} \cap \cdots \cap E_{i_{t}}}.$$ 

Thus
$$E^{t,-1}_{1}(1)={\rm CH}^{1}(E^{[t]}, 1)=\Gamma(E^{[t]}, \mathcal{O}^{\ast}_{E^{[t]}})=\bigoplus_{\{i_{0}, \ldots, i_{t}\} \in \Gamma} \mathbb{C}^{\ast} \cong C^{t}(\Gamma) \otimes_{\mathbb{Z}} \mathbb{C}^{\ast},$$ where $C^{t}(\Gamma)=\mathbb{Z}^{\#\{i_{1}, \ldots, i_{t} : E_{i_{1}, \ldots, i_{t}}\}}$. The complex $E^{\bullet, -1}_{1}(1)$ associated to $E$, is given by
\[ \begin{tikzcd}
0 \ar{r} & {\rm CH}^{1}(E^{[0]}, 1) \ar{r}{d^{\ast}} & {\rm CH}^{1}(E^{[1]}, 1) \ar{r}{d^{\ast}} & \cdots \ar{r}{d^{\ast}} & {\rm CH}^{1}(E^{[d-2]}, 1) \ar{r}{d^{\ast}} & {\rm CH}^{1}(E^{[d-1]}, 1) \ar{r} & 0
\end{tikzcd} \] where the differentials $d^{\ast}$ are the alternating sums of pull-backs of the face maps. Therefore, it is equivalent to the chain complex
\[ \begin{tikzcd}
0 \ar{r} & C^{0}(\Gamma) \otimes_{\mathbb{Z}} \mathbb{C}^{\ast} \ar{r}{\delta \otimes 1} & C^{1}(\Gamma) \otimes_{\mathbb{Z}} \mathbb{C}^{\ast} \ar{r}{\delta \otimes 1} & \cdots \ar{r}{\delta \otimes 1} & C^{d-2}(\Gamma) \otimes_{\mathbb{Z}} \mathbb{C}^{\ast} \ar{r}{\delta \otimes 1} & C^{d-1}(\Gamma) \otimes_{\mathbb{Z}} \mathbb{C}^{\ast} \ar{r} & 0
\end{tikzcd} \] associated to the dual complex $\Gamma(E)$.
\end{proof}

Returning to the case of three-dimensional varieties, we have the following result.

\begin{pro}\label{norma}
Let $E$ be a connected simple normal crossing divisor with irreducible intersections on a variety $X$ of dimension 3, with associated dual complex $\Gamma(E)$ such that $H^{2}(\Gamma)=0$. Then ${\rm CHC}^{1}(E,m)=0$ for $m\neq -1,0,1$, ${\rm CHC}^{1}(E,1)=\mathbb{C}^{\ast}$, and there is an exact sequence:
\[ \begin{tikzcd}
0 \ar{r} & H^{1}(\Gamma) \otimes \mathbb{C}^{\ast} \ar{r} & {\rm CHC}^{1}(E) \ar{r} & E^{0,0}_{2}(1) \ar{r} & 0.
\end{tikzcd} \]
\end{pro}
\begin{proof} 
Consider the semi-simplicial scheme $e\colon E^{[\bullet]} \rightarrow E$ augmented over $E$, then we have the fourth quadrant spectral sequence: 
$$E^{t,m}_{1}(1):={\rm CH}^{1}(E^{[t]}, -m) \ \Longrightarrow \ {\rm CHC}^{1}(E, -t-m).$$ 
	
The $E_1$-page of this spectral sequence can be seen as
\[ \begin{tikzcd}
E^{0,0}_{1}(1) \ar{r}{d^{0,0}_{1}} & E^{1,0}_{1}(1) \ar{r} & 0 \ar{r} & 0 \\
E^{0,-1}_{1}(1) \ar{r}{d^{0,-1}_{1}} & E^{1,-1}_{1}(1) \ar{r}{d^{1,-1}_{1}} & E^{2,-1}_{1}(1) \ar{r} & 0
\end{tikzcd} \] here $E^{2,0}_{1}(1)={\rm CH}^{1}(E^{[2]})=0$ for codimension reasons, because ${\rm dim} \ E^{[2]}=0$ and $E^{[t]}=\emptyset$ for $t > 2$. The differentials are given by alternating sums of pull-backs of the face maps. More explicitly
\[ \begin{tikzcd}
{\rm CH}^{1}(E^{[0]})  \ar{r}{d^{0,0}_{1}} \ar[dashed]{drr}{d^{0,0}_{2}} & {\rm CH}^{1}(E^{[1]}) \ar{r} & 0 \ar{r} & 0 \\
{\rm CH}^{1}(E^{[0]}, 1) \ar{r}[swap]{d^{0,-1}_{1}} & {\rm CH}^{1}(E^{[1]}, 1) \ar{r}[swap]{d^{1,-1}_{1}} & {\rm CH}^{1}(E^{[2]}, 1) \ar{r} & 0
\end{tikzcd} \] 

By Lemma \ref{lem}, the complex $E^{\bullet, -1}_{1}(1)$ of three terms is equivalent to the cochain complex that computes the cohomology of $\Gamma(E)$. Then, the bottom part of the $E_{2}$-page is the cohomology of the dual complex, i.e., $E^{t, -1}_{2}(1) \cong H^{t}(\Gamma(E), \mathbb{Z}) \otimes \mathbb{C}^{\ast}$. There is a differential that is possibly non-zero on the $E_{2}$-page, that is, $d^{0,0}_{2}\colon {\rm Ker}(d^{0,0}_{1}) \rightarrow {\rm Coker}(d^{1,-1}_{1})$. But $H^{2}(\Gamma)=0$, and this implies that $d^{1,-1}_{1}$ is a surjective map, so $d^{0,0}_{2}=0$ and we have $E^{2,-1}_{2}=E^{2,-1}_{\infty}$. Then, the spectral sequence reduces to a short exact sequence of the form:
\[\begin{tikzcd}
0 \ar{r} & E^{2,-1}_{2}(1) \ar{r} & {\rm CHC}^{1}(E, -1) \ar{r} & E^{1, 0}_{2}(1) \ar{r} & 0
\end{tikzcd}\] 

But $E^{2, -1}_{2}(1)={\rm Coker}(d^{1,-1}_{1})=0$, so that ${\rm CHC}^{1}(E, -1) \cong E^{1, 0}_{2}(1)={\rm Coker}\big({\rm CH}^{1}(E^{[0]})  \rightarrow {\rm CH}^{1}(E^{[1]})\big)$. In degree zero we have the following exact sequence:
\[\begin{tikzcd}
0 \ar{r} & E^{1,-1}_{2}(1) \ar{r} & {\rm CHC}^{1}(E) \ar{r} & E^{0, 0}_{2}(1) \ar{r} & 0
\end{tikzcd}\] 

which are the only terms that are not zero. This completes the proof.
\end{proof} 


The complete semi-simplicial hyperresolution of $X$ is given by
\[ \begin{tikzcd}[column sep=2em]
	E^{[2]} \ar[shift left=0.6em]{r} \ar{r} \ar[shift right=0.6em]{r} & E^{[1]} \ar[shift left=0.3em]{r} \ar[shift right=0.3em]{r} & E^{[0]} \ar{d} \\ 
	& & E \ar[shift left=0.3em]{r} \ar[shift right=0.3em]{r} & X_{0}=\widetilde{X} \sqcup X_{\rm sing} \ar{r} & X \\ X_{3}=E^{[2]} \ar[shift left=0.9em]{r}\ar[shift left=0.3em]{r}\ar[shift right=0.3em]{r}\ar[shift right=0.9em]{r} & X_{2}=E^{[1]} \ar[shift left=0.6em]{r} \ar{r} \ar[shift right=0.6em]{r} & X_{1}=E^{[0]} \ar[shift left=0.3em]{r} \ar[shift right=0.3em]{r} & X_{0}=\widetilde{X} \sqcup X_{\rm sing} \ar{r} & X
\end{tikzcd} \]

\begin{tikzpicture}[overlay, remember picture]
	\draw[thick] (2.0, 1.29) -- (14.3, 1.29); 
\end{tikzpicture} 

\begin{theo}\label{te}
Let $X$ be an irreducible projective variety of dimension 3 over $\mathbb{C}$, with isolated singularities. Suppose that $H^{2}(\Gamma(E))=0$ for the dual complex $\Gamma(E)$ associated to the simple normal crossing divisor $E$. Then ${\rm CHC}^{1}(X, m)=0$ for $m \neq -2,-1, 0, 1$, ${\rm CHC}^{1}(X, 1)=\mathbb{C}^{\ast}$, and there is an exact sequence
\[ \begin{tikzcd}
0 \ar{r} & {\rm CHC}^{1}(X) \ar{r} & {\rm CH}^{1}(\widetilde{X}) \ar{r} & {\rm CHC}^{1}(E) \ar{r} & {\rm CHC}^{1}(X,-1) \ar{r} & 0 
\end{tikzcd} \]
\end{theo}
\begin{proof}
By Proposition \ref{norma} and previous results, the long exact sequence \eqref{eq1} can be reduced to the following exact sequence:
\[ \begin{tikzcd}[row sep=0.6em]
0 \ar{r} & {\rm CHC}^{1}(X,1) \ar{r} & (\mathbb{C}^{\ast})^{\oplus 2} \ar{r} & \mathbb{C}^{\ast} \ar{r} & {\rm CHC}^{1}(X) \ar{r} & {\rm CH}^{1}(\widetilde{X}) \\ \ar{r} & {\rm CHC}^{1}(E) \ar{r} & {\rm CHC}^{1}(X,-1) \ar{r} & 0 \ar{r} & {\rm CHC}^{1}(E,-1) \ar{r} & {\rm CHC}^{1}(X,-2) \ar{r} & 0
\end{tikzcd} \] 

Since the higher Chow groups of smooth varieties vanish in negative degrees, and ${\rm CH}^{1}(X_{\rm sing})=0$, the corollary follows immediately. 
\end{proof} 

Examples of this type of varieties can be seen in \cite[Theorem 2.2]{Step}.

\subsection{Varieties of higher dimension}
\noindent Let $X$ be a complex projective algebraic variety of dimension $d$ with isolated singularities $X_{\rm sing}$. In this case we have that
$${\rm CH}^{1}(X_{\rm sing}, 1)=\Gamma(X_{\rm sing}, \mathcal{O}^{\ast}_{X_{\rm sing}}) = (\mathbb{C}^{\ast})^{\oplus \ \# (X_{\rm sing})}.$$

Let $p\colon \widetilde{X} \rightarrow X$ be a resolution of singularities of $X$ such that $E=p^{-1}(X_{\rm sing})$ is smooth. Then, the semi-simplicial hyperresolution of $X$ is given by the following diagram:
\[ \begin{tikzcd}[column sep=2em]
	X_{1}=E \ar[shift left=0.3em]{r} \ar[shift right=0.3em]{r} & X_{0}=\widetilde{X} \sqcup X_{\rm sing} \ar{r} & X.
\end{tikzcd} \] 

Thus, the spectral sequence $E^{p,q}_{1}(1):={\rm CH}^{1}(X_{p}, -q) \Rightarrow {\rm CHC}^{1}(X, -p-q)$ defines a complex
\[ \begin{tikzcd}
	E^{\bullet, q}_{1}(1) \ \colon \ 0 \ar{r} & {\rm CH}^{1}(X_{0}, -q) \ar{r} & {\rm CH}^{1}(X_{1}, -q) 
\end{tikzcd} \] for $q=-1, 0$. Then we have a long exact sequence
\[ \begin{tikzcd}[row sep=0.6em]
	0 \ar{r} & {\rm CHC}^{1}(X, 1) \ar{r} & \Gamma(\widetilde{X}, \mathcal{O}^{\ast}_{\widetilde{X}}) \oplus \Gamma(X_{\rm sing}, \mathcal{O}^{\ast}_{X_{\rm sing}}) \ar{r} & \Gamma(E, \mathcal{O}^{\ast}_{E}) \ar{r} & {\rm CHC}^{1}(X) \\ \ar{r} & {\rm CH}^{1}(\widetilde{X}) \ar{r} & {\rm CH}^{1}(E) \ar{r} & {\rm CHC}^{1}(X, -1) \ar{r} & 0.
\end{tikzcd} \]

\noindent In the general case, if $E$ is a simple normal crossing divisor (non-smooth), the cohomological cycle complex is given by
$$\mathcal{Z}^{r}(X, \bullet)^{\ast}:={\rm Cone}^{\bullet}\big[\mathcal{Z}^{r}(\widetilde{X}, \bullet) \oplus \mathcal{Z}^{r}(X_{\rm sing}, \bullet) \ \longrightarrow \ \mathcal{Z}^{r}(E, \bullet)^{\ast}\big][-1].$$ 

Recall that by induction on the hyperresolution, we can see that ${\rm CHC}^{1}(X, m)=0$ if $m \geq 2$ by Proposition \ref{codim}. In particular, we have a long exact sequence
\[ \begin{tikzcd}[row sep=0.6em]
0 \ar{r} & {\rm CHC}^{1}(X, 1) \ar{r} & {\rm CH}^{1}(\widetilde{X}, 1) \oplus {\rm CH}^{1}(X_{\rm sing}, 1) \ar{r} & {\rm CHC}^{1}(E, 1)  \ar{r} & {\rm CHC}^{1}(X) \\ \ar{r} & {\rm CH}^{1}(\widetilde{X}) \ar{r} & {\rm CHC}^{1}(E) \ar{r} & {\rm CHC}^{1}(X, -1) \ar{r} & 0 \\ \ar{r} & {\rm CHC}^{1}(E, -1) \ar{r} & {\rm CHC}^{1}(X, -2) \ar{r} & 0 \ar{r} & \cdots \\ \ar{r} & {\rm CHC}^{1}(E, -d+2) \ar{r} & {\rm CHC}^{1}(X, -d+1) \ar{r} & 0
\end{tikzcd} \] 

Thus we have ${\rm CHC}^{1}(E, m) \cong {\rm CHC}^{1}(X, m-1)$ for $m < 0$. In general ${\rm CHC}^{1}(X, m)=0$ if $1 > {\rm dim} \ X +m$ by Proposition \ref{dim}. Then we have to see where the groups ${\rm CHC}^{1}(E, m)$ are concentrated. Let us consider the semi-simplicial hyperresolution $E^{[\bullet]} \rightarrow E$, associated to the simple normal crossing divisor $E$. Then $E^{[t]}=\emptyset$ for $t > {\rm dim} \ X - 1$, otherwise ${\rm dim} \ E^{[t]}={\rm dim} \ X - 1 - t$. The spectral sequence associated is
$$E^{t, m}_{1}(1):={\rm CH}^{1}(E^{[t]}, -m) \Rightarrow {\rm CHC}^{1}(E, -t-m)$$ where $E^{\bullet, m}_{1}(1)$ is the complex ${\rm CH}^{1}(E^{[\bullet]}, -m)$; explicitly
\[ \begin{tikzcd}
	E^{\bullet, m}_{1}(1) \colon 0 \ar{r} & {\rm CH}^{1}(E^{[0]}, -m) \ar{r} &  {\rm CH}^{1}(E^{[1]}, -m) \ar{r} & {\rm CH}^{1}(E^{[2]}, -m) \ar{r} & \cdots
\end{tikzcd} \] for $m=0, -1$. The $E_{1}$- page of the spectral sequence is
\[ \begin{tikzcd}
	E^{0, 0}_{1} \ar{r}{d^{0,0}_{1}} \ar[dashed]{drr}{d^{0,0}_{2}} & E^{1, 0}_{1} \ar{r}{d^{1,0}_{1}} & E^{2, 0}_{1} \ar{r}{d^{2,0}_{1}} & \dots \ar{r} & E^{d-3, 0}_{1} \ar{r}{d^{d-3,0}_{1}} \ar[dashed]{drr}{d^{d-3,0}_{2}} & E^{d-2, 0}_{1} \ar{r} & 0 \ar{r} & 0 \\
	E^{0, -1}_{1} \ar{r}[swap]{d^{0,-1}_{1}} & E^{1, -1}_{1} \ar{r}[swap]{d^{1,-1}_{1}} & E^{2, -1}_{1} \ar{r}[swap]{d^{2,-1}_{1}} & \dots \ar{r} & E^{d-3, -1}_{1} \ar{r}[swap]{d^{d-3,-1}_{1}} & E^{d-2, -1}_{1} \ar{r}[swap]{d^{d-2,-1}_{1}} & E^{d-1, -1}_{1} \ar{r} & 0
\end{tikzcd} \] where the differentials on the second page are $d^{t,0}_{2}\colon E^{t,0}_{2}(1) \rightarrow E^{t+2,-1}_{2}(1)$, for $t=0, 1, \ldots, d-3$.

\begin{pro}\label{tee}
Let $E$ be a simple normal crossing divisor with irreducible intersections on a normal variety $X$ of dimension $d$, with associated dual complex $\Gamma(E)$ contractible. Then:
\begin{displaymath}
	{\rm CHC}^{1}(E,m)= \left\{ \begin{array}{lll}
		\mathbb{C}^{\ast}, & \quad {\rm if} \ m=1 \\
		H^{-m}({\rm CH}^{1}\big(E^{[\bullet]})\big), & \quad {\rm if} \ -d+2 \leq m \leq 0 \\
		0, & \quad {\rm otherwise}.
	\end{array} \right. 
\end{displaymath}
\end{pro}
\begin{proof}
By Lemma \ref{lem}, the $E_{1}$-page of the spectral sequence is given by:
\[ \begin{tikzpicture}
	\draw[->] (-0.7,0) -- (15,0) node[right] {$t$};
	\draw[->] (0,-1) -- (0,1.5) node[above] {$m$};
	
	\node (H0) at (1,0.6) {${\rm CH}^{1}(E^{[0]})$};
	\node (H1) at (3.7,0.6) {${\rm CH}^{1}(E^{[1]})$};
	\node (H) at (6,0.6) {$\cdots$};
	\node (Hdd) at (8.5,0.6) {${\rm CH}^{1}(E^{[d-2]})$};
	\node (0) at (11.5,0.6) {$0$};
	\node (00) at (13.5,0.6) {$0$};
	\node (HzeroLeft) at (-0.3,0.6) {$0$};
	
	\draw[->] (H0) -- (H1);
	\draw[->] (H1) -- (H);
	\draw[->] (H) -- (Hdd);
	\draw[->] (Hdd) -- (0);
	\draw[->] (0) -- (00);
	
	\node (tminus1) at (-0.3,-0.6) {$-1$};
    	\node (C0) at (1,-0.6) {$C^{0}(\Gamma) \otimes \mathbb{C}^{\ast}$};
    \node (C1) at (3.7,-0.6) {$C^{1}(\Gamma) \otimes \mathbb{C}^{\ast}$};
    \node (C) at (6,-0.6) {$\cdots$};
    \node (Cdd) at (8.5,-0.6) {$C^{d-2}(\Gamma) \otimes \mathbb{C}^{\ast}$};
    \node (Z) at (11.5,-0.6) {$C^{d-1}(\Gamma) \otimes \mathbb{C}^{\ast}$};
    \node (Zz) at (13.5,-0.6) {$0$};
    
    \draw[->] (C0) -- (C1);
    \draw[->] (C1) -- (C);
    \draw[->] (C) -- (Cdd);
    \draw[->] (Cdd) -- (Z);
    \draw[->] (Z) -- (Zz);
\end{tikzpicture} \] and $E^{t,0}_{2}(1)$ is just the cohomology in degree $t$ of the complex $E^{\bullet, 0}_{1}(1)={\rm CH}^{1}(E^{[\bullet]})$, i.e., $E^{t,0}_{2}(1)=H^{t}({\rm CH}^{1}(E^{[\bullet]}))$ for $t=0, 1, \ldots, d-2$. Since $\Gamma(E)$ is contractible, the complex $E^{\bullet, -1}_{2} \cong H^{t}(\Gamma(E), \mathbb{Z}) \otimes \mathbb{C}^{\ast}$ is given by:
\[ \begin{tikzcd}
E^{0, -1}_{2}(1) \ar{r} & E^{1, -1}_{2}(1) \ar{r} & \cdots \ar{r} & E^{d-2, -1}_{1}(1) \ar{r} & E^{d-2, -1}_{1}(1) \ar{r} & 0
\end{tikzcd} \] is exact, and $E^{t, -1}_{2}(1) \neq 0$ only in $t=0$. Then, the exact sequence:
\[ \begin{tikzcd}
E^{d-3, 0}_{2}(1) \ar{r} & E^{d-1, -1}_{2}(1) \ar{r} & {\rm CHC}^{1}(E, -d+2) \ar{r} & E^{d-2, 0}_{2}(1) \ar{r} & 0
\end{tikzcd} \] induces an isomorphism ${\rm CHC}^{1}(E, -d+2) \cong E^{d-2, 0}_{1}(1)$, and in general ${\rm CHC}^{1}(E, m) \cong H^{-m}\big({\rm CH}^{1}(E^{[\bullet]})\big)$. Finally, ${\rm CHC}^{1}(E, 1) \cong E^{0,-1}_{2}(1)=\mathbb{C}^{\ast}$.
\end{proof}

The condition considered in this Proposition is different from the one used in Proposition \ref{norma}. There, the hypothesis is slightly weaker, since it only considers that $H^{2}(\Gamma)=0$ instead of being contractible. Although here it is also sufficient to require that the homology is zero, in some cases it is equivalent to requiring that the complex $\Gamma(E)$ is contractible \cite[Theoren 3.1]{Step}. Finally, we have the following result:

\begin{theo}\label{gen}
Let $X$ be a complex projective variety of dimension $d$, with isolated singularities $X_{\rm sing}$. Let $p\colon \widetilde{X} \rightarrow X$ be a resolution of singularities, such that $E=p^{-1}(X_{\rm sing})$ the simple normal crossing divisor with irreducible intersections and associated dual complex $\Gamma(E)$ contractible. Then ${\rm CHC}^{1}(X,m) \neq 0$ for $m=1, 0, -1, \ldots, d-1$, ${\rm CHC}^{1}(X, 1)=\mathbb{C}^{\ast}$, and there is an exact sequence:
\[ \begin{tikzcd}
0 \ar{r} & {\rm CHC}^{1}(X) \ar{r} & {\rm CH}^{1}(\widetilde{X}) \ar{r} & {\rm CHC}^{1}(E) \ar{r} & {\rm CHC}^{1}(X, -1) \ar{r} & 0.
\end{tikzcd} \] 
\end{theo}

\section{Acknowledgements}
\noindent I would like to begin by thanking Pedro Luis del Ángel for introducing me to these topics. To Alexey Beshenov and Omar Antolín-Camarena for the conversations and suggestions they gave me during the final stages of this work. I would also like to thank everyone who helped me along the way: Jaime Hernández, Daniel Duarte, and Eric Yen-Yo Chen. Finally, I would like to thank SECIHTI and its ``Estancias Posdoctorales por México" program for giving me the opportunity to continue my research. I would also like to thank the ``Instituto de Matemáticas-UNAM" for its hospitality.

\bibliographystyle{alpha}
\bibliography{Reff2}

\end{document}